\documentclass[11pt]{article}
\usepackage{amsmath}
\usepackage{dsfont}
\usepackage{mathrsfs}
\usepackage{amsmath,amssymb}
\usepackage{amsfonts}
\usepackage{hyperref}
\usepackage{amsthm}
\usepackage{graphicx}
\usepackage{subfigure}
\usepackage{xcolor}
\usepackage{cite}
\usepackage{epic,eepic,epsf,epsfig}
\usepackage{floatflt}
\hfuzz=\maxdimen
\theoremstyle{plain}
\newtheorem{lemma}{Lemma}[section]
\newtheorem{definition}[lemma]{Definition}
\newtheorem{proposition}[lemma]{Proposition}
\newtheorem{theorem}[lemma]{Theorem}
\newtheorem{corollary}[lemma]{Corollary}

\newtheorem{question}{Question}

\numberwithin{equation}{section}

\DeclareFixedFont{\Acknowledgment}{OT1}{cmr}{bx}{n}{14pt}
\allowdisplaybreaks
\begin{document}

\title{\bf The character of ideal circle patterns}

\author{Chang Li, Aijin Lin*, Liangming Shen}
\date{}
\maketitle

\begin{abstract}
Let $S$ be an oriented closed surface with a cellular decomposition $\mathcal{D}$ and a weight $\Phi\in(0, \pi)$. It is crucial to determine when $S$ supports an ideal $\mathcal{D}$-type circle pattern $\mathcal{P}$ with the exterior intersection angles given by $\Phi$. Rivin, Bobenko-Springborn and Ge-Hua-Zhou provided perfect solutions and gave wonderful criteria for the existence and uniqueness of ideal circle patterns. However, all criteria established by Rivin, Bobenko-Springborn and Ge-Hua-Zhou are extremely difficult to verify for the given cellular decomposition $\mathcal{D}$ and the weight $\Phi$. In this paper, we introduce the character $\mathcal{L}(\mathcal{D},\Phi)$ depends only on the data of the weighted cellular decomposition $(\mathcal D, \Phi)$ on $S$, and give some quite simple criteria for the existence of ideal circle patterns realizing $(\mathcal{D},\Phi)$. It seems that our character-type criteria are the first conditions totally different from criteria of Rivin, Bobenko-Springborn and Ge-Hua-Zhou, and provide more easily verifiable criteria. Our new character-type theorems may be of some independent interest. As an application, we give a new descriptions of the curvature image set $\mathbf{K}(\mathbb{R}^N_{>0})$. To approach our results, we shall use the combinatorial Ricci flows with ideal circle patterns introduced by Ge-Hua-Zhou as a fundamental tool. The main difficulty in the proof of our results is to establish the compactness of the solution to the flows. To circumvent the difficulty, we borrow the techniques developed by Ge and his collaborators.


\end{abstract}
\vspace{12pt}
\let\thefootnote\relax \footnotetext {*Corresponding author}

\section{Introduction}
\hspace{14pt}
This is a continuation of our study on criteria for the existence of Thurston's circle patterns in \cite{GL, LLS}. We give new criteria for the existence of Thurston's circle patterns in the hyperbolic and Euclidean background geometres in \cite{GL, LLS}. In this paper we focus on criteria for the existence of ideal circle patterns.\par
\subsection{Background}
\hspace{14pt}
To study the geometry and topology of 3-manifolds, Thurston \cite{T} states the
circle pattern theorem regarding the existence and uniqueness of circle patterns on higher
genus surfaces with prescribed combinatorial type and non-obtuse exterior intersection
angles. Thurston also noted that the sphere version of his theorem followed from
Andreev's theorem \cite{Andreev1, Andreev2}. For the special case of vanishing exterior exterior intersection
angles, the consequence was due to Koebe \cite{Koebe}.\par
Let $S$ be an oriented closed surface with a cellular decomposition $\mathcal{D}$. Assume that
$\mu$ is a Riemannian metric with constant curvature on $S$. A circle pattern $\mathcal{P}$ on $(S, \mu)$ is a
collection of oriented circles. And $\mathcal{P}$ is called $\mathcal{D}$-type, if there exists a geodesic cellular
decomposition $\mathcal{D}(\mu)$ of $(S, \mu)$ with the following properties: (i) $\mathcal{D}(\mu)$ is isotopic to $\mathcal{D}$; (ii)
the vertices of $\mathcal{D}(\mu)$ coincide with the centers of the disks in $\mathcal{P}$. Let $V, E, F$ denote the sets
of vertices, edges and 2-cells of $\mathcal{D}$. In this paper, we mainly focus on these $D$-type circle
patterns $\mathcal{P}=\{C_{v}: v\in V\}$ such that $C_{v}$ and $C_{u}$ intersect with each other, whenever there
exists an edge $e=[v, u]\in E$. Under this condition we have the exterior intersection angle
$\Phi(e)\in[0, \pi)$ for every $e\in E$. For each $v\in V$, let $\mathbb{D}_{v}$ (resp. $D_{v}$) denote the open (resp.
closed) disk bounded by $C_{v}$. We call each connected component of the set $S\setminus \bigcup_{v\in V}\mathbb{D}_{v}$ an interstice. One refers to Stephenson's monograph \cite{Stephenson} for more backgrounds on circle patterns.\par
The connection between circle patterns and hyperbolic polyhedra is observed by
Thurston \cite{T}. Given a convex hyperbolic polyhedron $P$ in the hyperbolic
3-space $\mathbb{B}^{3}$, the boundaries of the oriented hyperbolic planes containing its faces is a circle
pattern on the sphere $\mathbb{B}^{3}$ with the dual combinatorial type of $P$, and the dihedral
angles between the faces are equal to the exterior intersection angles between the circles.
The readers can refer to the works of
Marden-Rodin \cite{Mar-Rod}, Colin de Verdi\`{e}re \cite{Colin}, Bowditch \cite{Bodw}, Rivin-Hodgson \cite{R-H}, Rivin \cite{Rivin1}\cite{Rivin2}, Bao-Bonahon \cite{Bao}, Bobenko-Springborn \cite{BS}, Leibon \cite{Leib}, Rousset \cite{Rou}, Schlenker \cite{Schl} and others for relevant results on circle patterns and
hyperbolic polyhedra. Thurston also posed a conjecture regarding the convergence of infinitesimal hexagonal tangent circle patterns to conformal mappings \cite{T2}, which was proved by Rodin-Sullivan \cite{RS}. From then on circle patterns have played significant roles in the study of low dimensional geometry and topology, complex analysis \cite{He1}-\cite{He3}, and various problems in combinatorics \cite{Schr1, Liu, Zhou1, Zhou2}, discrete and computational geometry \cite{BowersStephenson-mams, Stephenson, Dai}, minimal surfaces \cite{Bob1}, and many others.\par
A convex hyperbolic polyhedron in $\mathbb{B}^{3}$ is called \emph{ideal} if all its vertices are on the sphere
at infinity. Namely, it is the convex hull of a finite set of points (not all contained on a
plane) on the sphere at infinity. The following is a parallel concept for circle patterns.
\begin{definition} A $\mathcal{D}$-type circle pattern $P$ on $(S, \mu)$ is called \emph{ideal} if it satisfies the following
properties: (i) there exists a one-to-one correspondence between the interstices of $\mathcal{P}$ and
the 2-cells of $\mathcal{D}$; (ii) every interstice of $\mathcal{P}$ consists of a point.
\end{definition}

Let $\Phi$ be the exterior intersection angle function of an ideal circle pattern. Then we
have $\Phi(e)\in(0, \pi)$ for every $e\in E$. Meanwhile, the following condition is also necessary:
($B_1$) For any distinct edges $e_{1}, \cdots, e_{m}$ forming the boundary of a 2-cell of $\mathcal{D}$,
\begin{equation}\label{B1}
\sum_{i=1}^{m}\Phi(e_{i})=(m-2)\pi.
\end{equation}
Conversely, given a function $\Phi: E\rightarrow(0, \pi)$ satisfying ($B_1$), it is natural to consider
the following problem.
\begin{question}\label{question-existence}
 \item[]  Does there exist an ideal $\mathcal{D}$-type circle pattern whose exterior
intersection function is $\Phi$? And if it does, to what extent is the circle pattern unique?
\end{question}
Rivin \cite{Rivin1}, Bobenko-Springborn \cite{BS} and Ge-Hua-Zhou \cite{GHZ-2} provided perfect solutions to Question \ref{question-existence}, and gave wonderful criteria for the existence and uniqueness of ideal circle patterns. Especially, Ge-Hua-Zhou \cite{GHZ-2} pioneered the Chow-Luo's combinatorial Ricci flow method to study ideal circle patterns and gave a more systematic and comprehensive answer to Question \ref{question-existence}. Moreover, Ge-Hua-Zhou \cite{GHZ-2}  provided an algorithm (converging exponentially fast) to find the desired
ideal circle patterns.\par
In the case of $g=0$, the answer to Question \ref{question-existence} follows from Rivin's theorem on ideal hyperbolic
polyhedra \cite{Rivin1}. Given an abstract polyhedron $P$ with the dual polyhedron $P^{*}$, we
call a set $\{e_{1}, \cdots, e_{s}\}\subset P$ of edges a prismatic $s$-circuit, if their dual edges $\{e^{*}_{1}, \cdots, e^{*}_{s}\}$
form a simple closed path which does not bound a face of $P^{*}$. Rivin's theorem \cite{Rivin1} is then stated as follows.
\begin{theorem}\label{Rivin}(Rivin). Given an abstract polyhedron $P$, assume that $\Phi(e)\in(0, \pi)$ is a function
defined on its edge set. There exists a convex hyperbolic ideal polyhedron $Q$ combinatorially
equivalent to $P$ with the dihedral angles given by $\Phi$ if and only if the following conditions hold:\\
($A_1$). When $e_{1}, \cdots, e_{m}$ are all distinct edges meeting at a vertex,
$\sum_{i=1}^{m}\Phi(e_{l})=(m-2)\pi$;\\
($A_2$). When the edges $e_{1}, \cdots, e_{s}$ form a prismatic $s$-circuit,
$\sum_{i=1}^{s}\Phi(e_{l})<(s-2)\pi$.\\
Moreover, $Q$ is unique up to hyperbolic congruences.
\end{theorem}
In the case of $g>0$, the problem was solved by the following theorem due to
Bobenko-Springborn \cite{BS}. We say a closed (not necessarily simple) path $\gamma$
 in $S$ is pseudo-Jordan, if $S\setminus\gamma$ contains a simply-connected component with boundary $\gamma$.
\begin{theorem}\label{Bobenko}(Bobenko-Springborn).
Let $\mathcal{D}$ be a cellular decomposition of a compact oriented
surface $S$ of genus $g>0$. Assume that $\Phi(e)\in(0, \pi)$ is a function defined on its edge set. There
exists a constant curvature (equal to $0$ for $g=1$ and equal to $-1$ for $g>1$) metric $\mu$ on $S$ such
that $(S, \mu)$ supports an ideal $\mathcal{D}$-type circle pattern $\mathcal{P}$ with the exterior intersection angles given
by $\Phi$ if and only if ($B_1$) (\ref{B1}) and the following condition hold:\\
($B_2$). When the edges $e_{1}, \cdots, e_{s}$ form a pseudo-Jordan path which is not the boundary of a
2-cell of $\mathcal{D}$, $\sum_{i=1}^{s}\Phi(e_{l})<(s-2)\pi$.\\
Moreover, the pair $(\mu, \mathcal{P})$ is unique up to isometry if $g>1$, and up to similarity if $g=1$.
\end{theorem}
Recall that $\mathcal{D}$ is a cellular decomposition of $S$. For each 2-cell $\tau$ of $\mathcal{D}$, add a vertex into $\tau$ and add edges to connect this new vertex to those vertices of $\tau$. Doing this procedure
for all 2-cells, one obtains a triangulation $\mathcal{T}(\mathcal{D})$ of $S$. Let $V^{\diamond}, E^{\diamond}, F^{\diamond}$ denote the sets of vertices, edges and triangles of $\mathcal{T}(\mathcal{D})$. Then we have the following identification
$$V^{\diamond}=V\cup F, E^{\diamond}=E\cup E_{\tau v}$$
where $E_{\tau v}$ consists of the edges $[\tau, v]$ such that $\tau$ is a 2-cell of $\mathcal{D}$ and $v\in V$ is a vertex of $\tau$.
For each $v^{\diamond}\in V^{\diamond}$, we call it a \emph{primal vertex}, if $v^{\diamond}\in V$; Otherwise, we call it a \emph{star-vertex}.\par
Let us introduce Thurston's construction for circle patterns \cite{T}. Assume
that $V=\{v_{1},\ldots, v_{N}\}$, where $N =|V|$ is the number of vertices. We start with a radius vector $r=(r_1,\cdots, r_N)\in\mathbb{R}^N_{>0}$, which assigns a circle with radius $r_i$ at every vertex $i$, each vertex $v_{1}$ a positive number $r_i$. A radius vector $r$ produces a hyperbolic (or Euclidean)
cone metric on the surface $S$ as follows.

First, note that each triangle $\triangle(v_{i}v_{j}\tau_{k})$ of $\mathcal{T}(\mathcal{D})$ is marked by an edge $e_{ij}=[v_{i}, v_{j}]$ and a 2-
cell $\tau_{k}$ of $\mathcal{D}$. Because $\Phi_{ij}=\Phi(e_{ij})=\Phi([v_{i}, v_{j}])\in(0, \pi)$, there exists a pair of two intersecting
circles in hyperbolic (or Euclidean) geometry, unique up to isometry, having radii $r_{i}, r_{k}$
and meeting in angle $\Phi_{ij}$. One associates $\triangle(v_{i}v_{j}\tau_{k})$ with the hyperbolic (or Euclidean)
triangle whose vertices are the two centers and one of the intersection points of the above
two-circle configuration (see Section 3 for more details). More precisely, let $l_{ij}$, $l_{jk}$ and $l_{ki}$ denote the lengths of
the sides of this triangle. Then
$$l_{jk}=r_{i}, l_{ki}=r_{j}$$
and
\begin{equation}\label{length-H}
l_{ij} =\cosh^{-1}(\cosh r_{i}\cosh r_{j}+\sinh r_{i}\sinh r_{j}\cos\Phi(e_{ij}))
\end{equation}
in the hyperbolic background geometry, or
\begin{equation}\label{length-E}
l_{ij} =\sqrt{r_{i}^{2}+r_{j}^{2}+2r_{i}r_{j}\cos\Phi(e_{ij})}
\end{equation}
in the Euclidean background geometry.
Because $l_{ij}$, $l_{jk}$ and $l_{ki}$ satisfy the triangle inequalities, the
process works well.\par
Gluing all these hyperbolic (or Euclidean) triangles along their common edges then
produces a hyperbolic (or Euclidean) cone metric $\mu(r)$ on $S$ with possible singularities at
vertices of $\mathcal{T}(\mathcal{D})$. Let us investigate the cone angle. For any star-vertex $\tau\in V^{\diamond}$ which
represents a 2-cell bounded by the edges $e_{1}, \cdots, e_{m}\in E$, by ($B_1$) (\ref{B1}), the cone angle at $\tau$ is
$$\sum_{i=1}^{m}(\pi-\Phi(e_{i}))=m\pi-\sum_{i=1}^{m}\Phi(e_{i})=m\pi-(m-2)\pi=2\pi.$$
Therefore the gluing metric has no singularity at any star-vertex. It remains to deal with
the cone singularities at primal vertices. For each $v_{i}\in V$, the vertex curvature at $v_{i}$ is
defined by $K_{i}=2\pi-a_{i}$, where $a_{i}$ denotes the cone angle at $v_{i}$. Let $\theta^{j}_{i}, \theta^{i}_{j}$ denote the two corresponding inner angles at the centers $i, j$ in the $\triangle(v_{i}v_{j}\tau_{k})$. Then according to the construction above, the cone angle $a_{i}=2\sum_{j\sim i}\theta^{j}_{i}$ at $v_{i}$ (see Section 3 for more details). Therefore, the vertex curvature $K_{i}$ at $v_{i}$ can be expressed as follows.
\begin{equation}\label{vertex-curvature}
K_{i}=2\pi-2\sum_{j\sim i}\theta^{j}_{i}
\end{equation}
Consider the following curvature map
 \begin{align*}
 \mathbf{K} : \mathbb{R}^N_{>0}&\rightarrow \mathbb{R}^N\\
     r=(r_1,\cdots, r_N)&\mapsto K=(K_1,\cdots, K_N).
 \end{align*}
It is important to show that the origin $(0,\cdots, 0)$ belongs to the image. If
there exists a radius vector $r^{*}$ such that $K_{i}(r^{*})=0$ for $i = 1,\cdots, N$, then it produces
a smooth hyperbolic (or Euclidean) metric on the compact surface $S$. Drawing the circle centered at $v_{i}$ with radius $r^{*}$, one will obtain the desired ideal circle pattern realizing $(\mathcal{D}, \Phi)$. To this end, a key step is to give a description of the image of the curvature map. This is settled
by Bobenko-Springborn \cite{BS} via a variational principle.
\begin{theorem}\label{Bobenko-Springborn-1}(Bobenko-Springborn).
Let $S$ be an oriented closed surface with a cellular decomposition $\mathcal{D}$. Assume that $\Phi: E\rightarrow (0, \pi)$ satisfies the condition ($B_1$) (\ref{B1}. In the hyperbolic background geometry, $\mathbf{K}$ is injective. Moreover, the image of this map consists of vectors $(K_1, \cdots, K_N)$ satisfying
\begin{equation*}\label{con-1}
K_{i}< 2\pi, i=1,\ldots, N,
\end{equation*}
and
\begin{equation*}\label{con-2}
\sum_{i\in A}K_{i}> 2\pi|A|-2\sum_{e\in E, \partial e\cap A\neq \emptyset}\Phi(e)
\end{equation*}
for any non-empty subset $A\subset V$, where $\partial e$ denotes the two end vertices of an edge $e$.
\end{theorem}

\begin{theorem}\label{Bobenko-Springborn-2}(Bobenko-Springborn).
Let $S$ be an oriented closed surface with a cellular decomposition $\mathcal{D}$ and a weight $\Phi\in(0, \pi)$. In the Euclidean background geometry, $\mathbf{K}$ is injective up to scaling. Moreover, the image of this map consists of vectors $(K_1, \cdots, K_N)$ satisfying
\begin{equation*}
K_{i}< 2\pi, i=1,\ldots, N,
\end{equation*}
and
\begin{equation*}
\sum_{i\in A}K_{i}\geq 2\pi|A|-2\sum_{e\in E, \partial e\cap A\neq \emptyset}\Phi(e)
\end{equation*}
for any non-empty subset $A\subset V$, where the equality holds if and only if $A =V$.
\end{theorem}

A more challenging problem is to search the desired ideal circle patterns. For this purpose, Ge-Hua-Zhou \cite{GHZ-2} considered Chow-Luo's combinatorial Ricci flows, namely discrete versions of the Ricci flow which was introduced by Hamilton \cite{RH}. As a powerful tool, the Ricci flow has been used to study geometric structures on a given manifold, which is extremely important in differential geometry. Especially, Perelman \cite{Pe} proved the Poincar$\acute{\textrm{e}}$ conjecture and Thurston's geometrization conjecture using the Ricci flow. In addition, there have been many recent profound results such as Jiang-Naber \cite{JN} and Cheeger-Jiang-Naber \cite{CJN} on geometric structures related to the Ricci curvature. \par
The combinatorial Ricci flows which are the analogue of Hamilton's Ricci flow in the combinatorial setting were introduced by Chow-Luo \cite{CL}. They obtained a new proof of the famous Thurston's circle packings theorems, by showing that the combinatorial Ricci flows produce solutions which converge exponentially fast to Thurston's circle packings on surfaces. Since then, the combinatorial Ricci flows provided an effective tool for studying geometric structures on low dimensional manifolds. We refer the the readers to see \cite{FGH, Ge-thesis, Ge1, Ge2, GH, GHZ, GHZ-2, GJ, GJS, GL}. \par
Especially, Ge-Hua-Zhou \cite{GHZ-2} was the first to introduce the combinatorial Ricci flows with ideal circle patterns in the hyperbolic background geometry such as
\begin{equation}\label{equ-H}
\frac{dr_i}{dt}=-K_i\sinh r_i,
\end{equation}
and in the Euclidean background geometry such as
\begin{equation}\label{equ-E}
\frac{dr_i}{dt}=(K_{av}-K_i) r_i,
\end{equation}
for $i =1, \ldots, N$, with an initial radius vector $r(0)\in R^{N}_{>0}$, where $K_{av}=2\pi\chi(S)/N$.\par

Furthermore, Ge-Hua-Zhou \cite{GHZ-2} established the theory of combinatorial Ricci flows for ideal circle patterns which not only give a more systematic and comprehensive answer to Question \ref{question-existence}, but also provided an algorithm (converging exponentially fast) to find the desired
ideal circle patterns (and ideal hyperbolic polyhedra). Ge-Hua-Zhou's main results are stated as follows.
\begin{theorem}\label{Ge-Hua-Zhou-H}(Ge-Hua-Zhou).
Let $S$ be an oriented closed surface with a cellular decomposition $\mathcal{D}$. Assume that $\Phi: E\rightarrow (0, \pi)$ satisfies the condition ($B_1$) (\ref{B1}. In the hyperbolic background geometry, the solution $r(t)$ to the flow (\ref{equ-H}) exists for all the time $t\geq 0$, and the following ($H_1$)-($H_5$) are all equivalent:\\
($H_1$). $r(t)$ converges as $t\rightarrow +\infty$.\\
($H_2$). The origin $(0, \ldots, 0)$ belongs to the image of the curvature map.\\
($H_3$). When $A$ is a non-empty subset of $V$, the following condition holds:\\
\begin{equation*}
\sum_{e\in E, \partial e\cap A\neq \emptyset}\Phi(e)> \pi|A|.
\end{equation*}
($H_4$). When $A$ is a non-empty subset of $V$, the following condition holds:
\begin{equation*}
\sum_{(e, \tau) \in Lk(A)}(\Phi(e)-\pi)+2\pi\chi(S(A))<0.
\end{equation*}
($H_5$). The genus $g>1$ and the weight satisfies $\sum_{l=1}^{s}\Phi(e_{l})<(s-2)\pi$, whenever $e_1,\ldots, e_s$ form a \textbf{pseudo-Jordan} curve which is not the boundary of a 2-cell of $\mathcal{D}$.\\
Moreover, if the flow (\ref{equ-H}) converges, then it converges exponentially fast to the unique radius vector so that all vertex curvatures are equal to zero.
\end{theorem}
\begin{theorem}\label{Ge-Hua-Zhou-E}(Ge-Hua-Zhou).
Let $S$ be an oriented closed surface with a cellular decomposition $\mathcal{D}$ and a weight $\Phi\in(0, \pi)$. In the Euclidean background geometry, the solution $r(t)$ to the flow (\ref{equ-E}) exists for all the time $t\geq 0$, and the following ($E_1$)-($E_4$) are all equivalent:\\
($E_1$). $r(t)$ converges as $t\rightarrow +\infty$.\\
($E_2$). The constant vector $(K_{av}, \ldots, K_{av})$ belongs to the image of the curvature map.\\
($E_3$). When $A$ is a non-empty subset of $V$, the following condition holds:\\
\begin{equation*}
\frac{\pi\chi(S)|A|}{|V|}> \pi|A|-\sum_{e\in E, \partial e\cap A\neq \emptyset}\Phi(e).
\end{equation*}
($E_4$). When $A$ is a non-empty subset of $V$, the following condition holds:
\begin{equation*}
\frac{2\pi\chi(S)|A|}{|V|}> \sum_{(e, \tau) \in Lk(A)}(\Phi(e)-\pi)+2\pi\chi(S(A)).
\end{equation*}
Moreover, if the flow (\ref{equ-E}) converges, then it converges exponentially fast to a radius vector so that all vertex curvatures are equal to $K_{av}$.
\end{theorem}

\subsection{Main results}
\hspace{14pt}
The main results above of Rivin, Bobenko-Springborn and Ge-Hua-Zhou provide a wonderful criteria for the existence and uniqueness of ideal circle patterns. However, all criteria above (such as Rivin's criteria $A_1, A_2$, Bobenko-Springborn's criteria $B_1, B_2$ and Ge-Hua-Zhou's criteria $H_1, H_3-H_5; E_1, E_3-E_4$) are extremely difficult to verify, since the angle structure $\Phi$ and the combinatorial structure of $\mathcal{D}$ are coupled together and globally on $S$. It is natural to ask the following question.
\begin{question}\label{question-simpler}
 \item[]  Can one give more easily verifiable criteria than the criteria of Rivin, Bobenko-Springborn and Ge-Hua-Zhou to ensure the existence of the desired ideal circle patterns?
\end{question}
Our following ``character-type" theorems can give an answer to Question \ref{question-simpler}. First, inspired by the work in Ge-Jiang-Shen \cite{GJS}, Ge-Hua \cite{GH} and Feng-Ge-Hua \cite{FGH} on 3-dimensional geometric triangulations, we introduce the characters of ideal circle patterns. Let $S$ be an oriented closed surface with a cellular decomposition $\mathcal{D}$ and an exterior intersection function $\Phi(e)\in(0, \pi)$ defined on its edge set. Define the \emph{character} $\mathcal{L}(\mathcal{D},\Phi)_i$ at each vertex $i\in V$ by
\begin{equation*}
\mathcal{L}(\mathcal{D},\Phi)_i=\sum_{j\sim i}\Phi_{ij},
\end{equation*}
where $\Phi_{ij}$ denotes the value of $\Phi$ on the edge $e_{ij}\in E$, and $j\sim i$ denotes that the vertices $i$ and $j$ are adjacent.\par
Then we found simpler and more checkable criteria than criteria of Rivin, Bobenko-Springborn and Ge-Hua-Zhou for the existence of ideal circle patterns.
\begin{theorem}\label{thm-main-H}
Let $G$ be the 1-skeleton of a cellular decomposition $\mathcal{D}$ of a closed surface $S$ with a weight $\Phi\in(0, \pi)$. \\
(1). If the character $\mathcal{L}(\mathcal{D},\Phi)_i>2\pi$ at all vertices, then there
exists a constant curvature metric $\mu$ on $S$ that $(S, \mu)$ supports a ideal $G$-type circle pattern $\mathcal{P}$ with the intersection angle function given by $\Phi$. \\
(2). If the character $\mathcal{L}(\mathcal{D},\Phi)_i\leq2\pi$ at all vertices, then there exists no such ideal $G$-type circle patterns $\mathcal{P}$ with the intersection angle function given by $\Phi$.
\end{theorem}
The distinct feature of the characters $\mathcal{L}(\mathcal{D},\Phi)_i$ is their \emph{localness}, and they can be checked locally and separately at each vertex. However, criteria of Rivin, Bobenko-Springborn and Ge-Hua-Zhou (such as Rivin's criteria $A_1, A_2$, Bobenko-Springborn's criteria $B_1, B_2$ and Ge-Hua-Zhou's criteria $H_1, H_3-H_5; E_1, E_3-E_4$) are intertwined together, they are global and can not be verified separately.\par

More generally, we prove the following character theorems which cover Theorem \ref{thm-main-H} using the combinatorial Ricci flows (\ref{equ-H}) (\ref{equ-E}) with ideal circle patterns introduced by Ge-Hua-Zhou \cite{GHZ-2}.
\begin{theorem}\label{thm-character-H}
Let $S$ be an oriented closed surface with a cellular decomposition $\mathcal{D}$ and a weight $\Phi\in(0, \pi)$. Then for any initial circle pattern $r(0)\in\mathbb{R}_{>0}^{N}$, the combinatorial Ricci flow (\ref{equ-H}) has a unique solution $r(t)$ which exists for all time $t\geq0$. Moreover, we have\\
(1). If the character $\mathcal{L}(\mathcal{D},\Phi)_i>2\pi$ at each vertex, then there exists a unique hyperbolic ideal circle pattern (ideal circle pattern with zero curvature) $r_{ze}$, and the solution $r(t)$ to the combinatorial Ricci flow (\ref{equ-H}) converges exponentially fast to $r_{ze}$ when $t\rightarrow +\infty$. Moreover, the hyperbolic ideal circle pattern $r_{ze}$ determines a complete hyperbolic metric $\mu$ on $S$ such that $(S, \mu)$ supports a geometric decomposition isotopic to $\mathcal D$ so that the edges are geodesics.\\
(2). If the character $\mathcal{L}(\mathcal{D},\Phi)_i\leq2\pi$ at each vertex, then there exists no hyperbolic ideal circle patterns. Moreover, if the character $\mathcal{L}(\mathcal{D},\Phi)_i<2\pi$ at each vertex, the solution $r(t)$ to the combinatorial Ricci flow (\ref{equ-H}) converges exponentially fast to $0$ when $t\rightarrow +\infty$.
\end{theorem}
\begin{theorem}\label{thm-character-E}
Let $S$ be an oriented closed surface with a cellular decomposition $\mathcal{D}$ and a weight $\Phi\in(0, \pi)$. Then for any initial circle pattern $r(0)\in\mathbb{R}_{>0}^{N}$, the combinatorial Ricci flow (\ref{equ-E}) has a unique solution $r(t)$ which exists for all time $t\geq0$. Moreover, we have\\
(1). If the character $\mathcal{L}(\mathcal{D},\Phi)_i\geq2\pi(1-\frac{\chi(S)}{N})$ at each vertex, then there exists a unique (up to scaling) constant ideal circle pattern (ideal circle pattern with constant curvature) $r_{av}$, and the solution $r(t)$ to the combinatorial Ricci flow (\ref{equ-E}) converges exponentially fast to $r_{av}$ when $t\rightarrow +\infty$. \\
(2). If the character $\mathcal{L}(\mathcal{D},\Phi)_i<2\pi(1-\frac{\chi(S)}{N})$ at each vertex, then there exists no constant ideal circle patterns. Moreover, the solution $r(t)$ to the combinatorial Ricci flow (\ref{equ-E}) converges exponentially fast to $0$ when $t\rightarrow +\infty$.
\end{theorem}
Theorem \ref{thm-character-H} and Theorem \ref{thm-character-E} can be generalized to find circle patterns with other prescribed vertex curvatures by modifications of the flows. It is natural to consider the following prescribed circle pattern problem.
\begin{question}\label{question-prescribed}
 \item[] Given a prescribed vertex curvature $\bar{K}=(\bar{K_1}, \cdots, \bar{K_N})$, does there exist an ideal $\mathcal{D}$-type circle pattern $\bar{r}$ with curvature $K(\bar{r})=\bar{K}$ and exterior intersection function $\Phi$? And if it does, to what extent is the circle pattern unique?
\end{question}
Our following character-type theorems can give a comprehensive answer to Question \ref{question-prescribed}. First of all, we consider the prescribed combinatorial Ricci flow in the hyperbolic background geometry
\begin{equation}\label{eq-prescribe-H}
\frac{dr_i}{dt}=(\bar{K_i}-K_i)\sinh r_i
\end{equation}
and the prescribed combinatorial Ricci flow in the Euclidean background geometry
\begin{equation}\label{eq-prescribe-E}
\frac{dr_i}{dt}=(\bar{K_i}-K_i)r_i.
\end{equation}
We have the following character-type theorems on the prescribed circle pattern problem.
\begin{theorem}\label{prescribed-H}
Let $S$ be an oriented closed surface with a cellular decomposition $\mathcal{D}$ and a weight $\Phi\in(0, \pi)$. For any prescribed vertex curvature $\bar{K}=(\bar{K_1}, \cdots, \bar{K_N})\in (-\infty, 2\pi)^{N}$, the prescribed flow (\ref{eq-prescribe-H}) has a unique solution $r(t)$ which exists for all time $t\geq0$ for any initial circle pattern $r(0)\in\mathbb{R}_{>0}^{N}$. Moreover,\\
(1). If $2\pi-\mathcal{L}(\mathcal{D},\Phi)_i<\bar{K_i}<2\pi$ at each vertex $i\in V$, then the prescribed circle pattern problem has a unique solution, namely a unique ideal circle pattern $\bar{r}$ with curvature $K(\bar{r})=\bar{K}$, and the solution to the prescribed flow (\ref{eq-prescribe-H}) converges exponentially fast to $\bar{r}$ when $t\rightarrow +\infty$. \\
(2). If $\bar{K_i}\leq2\pi-\mathcal{L}(\mathcal{D},\Phi)_i$ at each vertex $i\in V$, then the prescribed circle pattern problem had no solutions and equivalently, there is not any circle pattern $\bar{r}$ with curvature $K(\bar{r})=\bar{K}$. Moreover, if $\bar{K_i}<2\pi-\mathcal{L}(\mathcal{D},\Phi)_i$ at each vertex $i\in V$, then the solution $r(t)$ to the prescribed flow (\ref{eq-prescribe-H}) converges exponentially fast to $0$ when $t\rightarrow +\infty$.
\end{theorem}

\begin{theorem}\label{prescribed-E}
Let $S$ be an oriented closed surface with a cellular decomposition $\mathcal{D}$ and a weight $\Phi\in(0, \pi)$. For any prescribed vertex curvature $\bar{K}=(\bar{K_1}, \cdots, \bar{K_N})\in (-\infty, 2\pi)^{N}$ satisfying
\begin{equation*}
\sum_{i=1}^{N}\bar{K}_{i}=2\pi\chi(S),
\end{equation*}
the prescribed flow (\ref{eq-prescribe-E}) has a unique solution $r(t)$ which exists for all time $t\geq0$ for any initial circle pattern $r(0)\in\mathbb{R}_{>0}^{N}$. Moreover,\\
(1). If $2\pi-\mathcal{L}(\mathcal{D},\Phi)_i\leq\bar{K_i}<2\pi$ at each vertex $i\in V$, then the prescribed circle pattern problem has a unique (up to scaling) solution, namely a unique (up to scaling) ideal circle pattern $\bar{r}$ with curvature $K(\bar{r})=\bar{K}$, and the solution to the prescribed flow (\ref{eq-prescribe-E}) converges exponentially fast to $\bar{r}$ when $t\rightarrow +\infty$.\\
(2). If $\bar{K_i}<2\pi-\mathcal{L}(\mathcal{D},\Phi)_i$ at each vertex $i\in V$, then the prescribed circle pattern problem had no solutions and equivalently, there is not any circle pattern $\bar{r}$ with curvature $K(\bar{r})=\bar{K}$. Moreover, the solution $r(t)$ to the prescribed flow (\ref{eq-prescribe-E}) converges exponentially fast to $0$ when $t\rightarrow +\infty$.
\end{theorem}
Theorem \ref{prescribed-H} and Theorem \ref{prescribed-E} provide a comprehensive answer to Question \ref{question-prescribed}, and further give a new observation of the image set $\mathbf{K}(\mathbb{R}^N_{>0})$ base our character-type criteria.
\begin{theorem}\label{image-K-H}.
Let $S$ be an oriented closed surface with a cellular decomposition $\mathcal{D}$ and a weight $\Phi\in(0, \pi)$. In the hyperbolic background geometry, for the image set $\mathbf{K}(\mathbb{R}^N_{>0})$ we have
\begin{equation}
\prod_{i\in V}\Big(2\pi-\mathcal{L}(\mathcal{D},\Phi)_i, \; 2\pi\Big)\subset \mathbf{K}\big(\mathbb{R}^N_{>0}\big)
\end{equation}
and
\begin{equation}
\prod_{i\in V}\Big(-\infty, \; 2\pi-\mathcal{L}(\mathcal{D},\Phi)_i\Big]\cap \mathbf{K}\big(\mathbb{R}^N_{>0}\big)=\emptyset.
\end{equation}
\end{theorem}
\begin{theorem}\label{image-K-E}.
Let $S$ be an oriented closed surface with a cellular decomposition $\mathcal{D}$ and a weight $\Phi\in(0, \pi)$. In the Euclidean background geometry, for the image set $\mathbf{K}(\mathbb{R}^N_{>0})$ we have
\begin{equation}
\bigg\{x:\sum_{i\in V} x_i=2\pi\chi(S)\bigg\}\bigcap\prod_{i\in V}\Big[2\pi-\mathcal{L}(\mathcal{D},\Phi)_i, \; 2\pi\Big)\subset \mathbf{K}\big(\mathbb{R}^N_{>0}\big)
\end{equation}
and
\begin{equation}
\prod_{i\in V}\Big(-\infty, \; 2\pi-\mathcal{L}(\mathcal{D},\Phi)_i\Big)\cap \mathbf{K}\big(\mathbb{R}^N_{>0}\big)=\emptyset.
\end{equation}
\end{theorem}

It seems that our character-type criteria are the first conditions totally different from criteria of Rivin, Bobenko-Springborn and Ge-Hua-Zhou. To approach our results, we shall use the combinatorial Ricci flows with ideal circle patterns introduced by Ge-Hua-Zhou \cite{GHZ-2} as a fundamental tool. We also borrow the techniques developed in \cite{GJS,GH,FGH} to control the flows. The main difficulty in the proof of our results is to establish the compactness of the solution to the flows. To circumvent the difficulty, we study thoroughly the geometry of the basic building blocks, i.e. two-circle configurations. Particularly, we establish some comparison principles for inner angles both at the longest and the shortest circle pattern components, which are key to establish the compactness of the solution to the flows.

This paper is organized as follows. In Section 2, we introduce the characters of ideal circle patterns and establish some comparison principles for interior angles of two-circle configurations. In Section 3 and 4, we study the combinatorial Ricci flow with ideal circle patterns by characters, and further prove Theorem \ref{thm-character-H} and Theorem \ref{thm-character-E} respectively. In the final Section 5, we prove Theorem \ref{prescribed-H} and Theorem \ref{prescribed-E} by further investigations into the prescribed combinatorial Ricci flow.\\[-8pt]

\noindent{\bf Acknowledgments.}
The first author is partially supported by National Natural Science Foundation of China (Grant No. 12301079), the Fundamental Research Funds for the Central Universities, and the Research Funds of Renmin University of China (Grant No. 24XNKJ01). The second author is partially supported by National Natural Science Foundation of China (Grant No. 12171480), Hunan Provincial Natural Science Foundation of China (Grant No. 2020JJ4658 and No. 2022JJ10059) and Scientific Research Program of NUDT (Grant No. JS2023-01 and No. IISF-C24001). The third author is partially supported by National Natural Science Foundation of China (Grant No. 11901553 and No. 12271016) and Fundamental Research Funds for the Central Universities. Finally, the authors would like to thank Huabin Ge and Ke Feng for valuable conversations and helpful suggestions.

\section{Preliminaries}
\subsection{Characters of circle patterns}
Let $S$ be an oriented closed surface with a cellular decomposition $\mathcal{D}$ and a weight $\Phi\in(0, \pi)$. Doing the procedure
for all 2-cells in subsection above, one obtains a triangulation $\mathcal{T}(\mathcal{D})$ of $S$. Then we define characters of circle patterns as follows.
\begin{definition}\label{definition-character}
Let $S$ be an oriented closed surface with a cellular decomposition $\mathcal{D}$ and a weight $\Phi\in(0, \pi)$. For any circle pattern based on $(S, \mathcal {D}, \Phi)$, define the character $\mathcal{L}(\mathcal D, \Phi)_i$ at each $i\in V$ by
\begin{equation}\label{def-character}
\mathcal{L}(\mathcal{D},\Phi)_i=\sum_{j\sim i}\Phi_{ij},
\end{equation}
where the sum runs over all vertices which are adjacent to the vertex $v_{i}$. The character of $(S, \mathcal {D}, \Phi)$ is defined as
$$\mathcal{L}(\mathcal D, \Phi)=(\mathcal{L}(\mathcal D, \Phi)_1,\cdots,\mathcal{L}(\mathcal D, \Phi)_N).$$
\end{definition}

It seems that the character is just defined for a weighted cellular decomposition $(\mathcal D, \Phi)$ on $S$, and has no relation to any circle patterns. However, both the combinatorial structure $\mathcal D$ and the angle structure $\Phi$ come from a circle pattern $\mathcal P$ on $S$. Recall the contact graph of a circle pattern $\mathcal P$ is a graph having a vertex for each circle, and having an edge between the vertices for every two intersected circles. $\mathcal P$ is called \emph{$(\mathcal D, \Phi)$-type} if the contact graph is isotopic to the 1-skeleton of $\mathcal D$, and for each edge $e\in E$ connecting two vertices $v_i$ and $v_j$, the two circles centered at $v_i$ and $v_j$ intersects with an exterior intersection angle $\Phi(e)$. Hence the weighted cellular decomposition $(\mathcal D, \Phi)$ records all the information of a $(\mathcal D, \Phi)$-type circle pattern $\mathcal P$, that is, the combinatorial structure given by $\mathcal D$ and the angle structure given by $\Phi$. Hence the character $(\mathcal{L}(\mathcal{D},\Phi)_i)_{i\in V}$ is indeed an invariant of all $(\mathcal D,\Phi)$-type circle patterns $\mathcal P$ on a closed surface $S$.

\begin{proposition}\label{prop-cha-r=1}
Given a weighted cellular decomposed surface $(S,\mathcal{D},\Phi)$. For each vertex $i\in V$, the character $\mathcal{L}(\mathcal{D},\Phi)_i$ is exactly the cone angle $a_i$ of a circle pattern $r=(1,\cdots,1)$ in the Euclidean background geometry.
\end{proposition}

\begin{proof}
Assume $r_i=1$ for each vertex $i\in V$. Then by (\ref{length-E}), $l_{ij}=\sqrt{2+2\cos\Phi_{ij}}$ for each edge $e_{ij}\in E$. It's easy to see
$\theta_{i}^{j}=\Phi_{ij}/2$. Then by the definition of the character (\ref{def-character}), we obtain $\mathcal{L}(\mathcal{D},\Phi)_i=2\sum_{j\sim i}\theta_{i}^{j}=a_i$.
\end{proof}
Similar to the combinatorial Gauss-Bonnet formula (See Proposition 3.1 in \cite{CL} for details), there exists an equality for the total vertex curvature.
\begin{proposition}\label{GB}
For every circle pattern $r$ on $(S, \mathcal{D}, \Phi)$, then the total vertex curvature satisfies the following equality:
\begin{equation}\label{GB}
\sum_{i=1}^{N}K_{i}=2\pi\chi(X)-\lambda \text{Area}(X),
\end{equation}
where $\lambda=-1, 0, 1$ correspond three geometries, i.e., the hyperbolic background geometry $\mathbb{H}^{2}$, the Euclidean background geometry $\mathbb{E}^{2}$, and the spherical background geometry $\mathbb{S}^{2}$.
\end{proposition}
\begin{proposition}\label{prop-euler-number}
Denote the average character by $\mathcal{L}_{av}=\sum_{i\in V}\mathcal{L}(\mathcal{D},\Phi)_i/N$, then
$$\mathcal{L}_{av}=2\pi\Big(1-\frac{\chi(S)}{N}\Big).$$
Consequently, $\mathcal{L}_{av}>2\pi$ if $\chi(S)<0$, $\mathcal{L}_{av}=2\pi$ if $\chi(S)=0$ and $\mathcal{L}_{av}<2\pi$ if $\chi(S)>0$.
\end{proposition}

\begin{proof}
Note $2\pi-\mathcal{L}(\mathcal{D},\Phi)_i$ is the vertex curvature $K_i$ of a particular ideal circle pattern $r\equiv1$ in the Euclidean background geometry by Proposition \ref{prop-cha-r=1}. The conclusion follows from the total curvature formula (\ref{GB}) in the Euclidean background geometry
\begin{equation}\label{GB-E}
\sum_{i\in V}K_i=2\pi\chi(S).
\end{equation}

\end{proof}

\subsection{Geometry of two-circle configurations}
\hspace{14pt}
In this subsection, we study thoroughly the geometry of the basic building blocks, i.e. two-circle configurations. For more details and illustration examples please refer to Section 2 and Fig.1 in \cite{GHZ-2}. \par
Let $S$ be an oriented closed surface with a cellular decomposition $\mathcal{D}$ which produces a triangulation $\mathcal{T}(\mathcal{D})$ of $S$ and a weight $\Phi: E\rightarrow(0, \pi)$.
\begin{lemma}[Lemma 2.1, \cite{GHZ-2}] \label{two circles}
Given $\Phi_k\in(0, \pi)$, any two positive numbers $r_i, r_j$, there exists a configuration of two intersecting circles in both the hyperbolic and Euclidean geometries, unique up to isometries, having radii $r_i, r_j$ and meeting in exterior intersection angle $\Phi_k$.
\end{lemma}
By Lemma \ref{two circles}, let $\triangle(v_{i}v_{j}\tau_{k})$ be the triangle whose vertices are the two centers and one of the intersection points of two circles. Let $\theta^{j}_{i}, \theta^{i}_{j}$ denote the two corresponding inner angles at the centers $i, j$. By the (hyperbolic) cosine law, we have

$$\cos\theta_{i}^{j}(r_{i}, r_{j})=\frac{l^{2}_{ij}+r^{2}_{i}-r^{2}_{j}}{2l_{ij}r_{i}} (\cos\theta_{i}^{j}(r_{i}, r_{j})=\frac{\cosh l_{ij}+\cosh r_{i}-\cosh r_{j}}{\sinh l_{ij}\sinh r_{i}}).$$
Ge-Hua-Zhou \cite{GHZ-2} established the following proposition (see Lemma 2.2. in \cite{GHZ-2}) on partial derivatives of the inner angle $\theta_{i}^{j}$.
\begin{proposition} \label{derivative}
For a a weighted cellular decomposed surface $(S,\mathcal{D},\Phi)$ whose weight satisfies $\Phi : E\rightarrow(0, \pi)$. In both the Euclidean and hyperbolic background geometries $\mathbb{E}^{2}$, $\mathbb{H}^{2}$,  then
(1). $\frac{\partial \theta_{i}^{j}}{\partial r_{i}}<0$;
(2). $\frac{\partial \theta_{i}^{j}}{\partial r_{j}}>0$, for any $j\sim i$.
\end{proposition}
\begin{proof}
Here we give a new proof by direct computations. Set $f(r_{i}, r_{j})=\cos\theta_{i}^{j}(r_{i}, r_{j})$.\par
We claim that $\frac{\partial f}{\partial r_{i}}>0$ and $\frac{\partial f}{\partial r_{j}}<0$ which are equivalent to $\frac{\partial \theta_{i}^{j}}{\partial r_{i}}<0$ and $\frac{\partial \theta_{i}^{j}}{\partial r_{j}}>0$ in both the Euclidean and hyperbolic background geometries $\mathbb{E}^{2}$, $\mathbb{H}^{2}$.\par

 In the Euclidean background geometry $\mathbb{E}^{2}$, due to
the edge length formula $l_{ij} =\sqrt{r_{i}^{2}+r_{j}^{2}+2r_{i}r_{j}\cos\Phi_{ij}}$, we have
\begin{eqnarray}\label{cos-angle-E}
f(r_{i}, r_{j})=\frac{r_{i}+r_{j}\cos\Phi_{ij}}{l_{ij}}
                                 =\frac{r_{i}+r_{j}\cos\Phi_{ij}}{\sqrt{r_{i}^{2}+r_{j}^{2}+2r_{i}r_{j}\cos\Phi_{ij}}}.
\end{eqnarray}
By taking the partial derivative with respect to $r_{i}$, we have
\begin{eqnarray}\label{partial de-i-E}
\frac{\partial f}{\partial r_{i}}=\frac{r^{2}_{j}\sin^{2}\Phi_{ij}}{l^{3}_{ij}}
                                 =\frac{r^{2}_{j}\sin^{2}\Phi_{ij}}{(r_{i}^{2}+r_{j}^{2}+2r_{i}r_{j}\cos\Phi_{ij})^{\frac{3}{2}}}>0,
\end{eqnarray}
and
\begin{eqnarray}\label{partial de-j-E}
\frac{\partial f}{\partial r_{j}}=-\frac{r_{i}r_{j}\sin^{2}\Phi_{ij}}{l^{3}_{ij}}
                                 =-\frac{r_{i}r_{j}\sin^{2}\Phi_{ij}}{(r_{i}^{2}+r_{j}^{2}+2r_{i}r_{j}\cos\Phi_{ij})^{\frac{3}{2}}}<0.
\end{eqnarray}
 In the case of hyperbolic background geometry $\mathbb{H}^{2}$, recall
the edge length formula $$l_{ij} =\cosh^{-1}(\cosh r_{i}\cosh r_{j}+\sinh r_{i}\sinh r_{j}\cos\Phi_{ij}),$$ by direct computations we have
\begin{eqnarray}\label{partial de-i-H}
\frac{\partial f}{\partial r_{i}}=\frac{\sinh^{2}r_{j}\cosh l_{ij}\sin^{2}\Phi_{ij}}{\sinh^{3}l_{ij}}>0,
\end{eqnarray}
and
\begin{eqnarray}\label{partial de-j-H}
\frac{\partial f}{\partial r_{j}}=-\frac{\sinh r_{i}\sinh r_{j}\sin^{2}\Phi_{ij}}{\sinh^{3}l_{ij}}<0.
\end{eqnarray}
Hence we prove the claim and complete the proof of Proposition \ref{derivative}.
\end{proof}

Furthermore, when $(r_{i}, r_{j})=(t, t)$, we can get
\begin{proposition} \label{lim}
For a a weighted cellular decomposed surface $(S,\mathcal{D},\Phi)$ whose weight satisfies $\Phi : E\rightarrow(0, \pi)$. In the case of the Euclidean background geometry, then for any $t>0$,
$$\theta_{i}^{j}(t, t)\equiv \frac{\Phi_{ij}}{2}.$$
In the case of the hyperbolic background geometry, the inner angle $\theta_{i}^{j}(t, t)$ is continuously differentiable and strictly decreasing function in $(0, +\infty)$. Moreover, the following limits hold.
$$\textmd{lim}_{t\rightarrow 0}\theta_{i}^{j}(t, t)=\frac{\Phi_{ij}}{2}, \textmd{lim}_{t\rightarrow+\infty}\theta_{i}^{j}(t, t)=0.$$
\end{proposition}
Combining Proposition \ref{derivative} and Proposition \ref{lim}, we can get the following comparison principles.
\begin{theorem} \label{comparison}
In both the Euclidean and hyperbolic background geometries $\mathbb{E}^{2}$, $\mathbb{H}^{2}$, for the inner angle $\theta_{i}^{j}(r_{i}, r_{j})$ at the vertex $i$, we have \\
(1). if $r_{i}\leq r_{j}$, then $\theta_{i}^{j}(r_{i}, r_{j})\geq\theta_{i}^{j}(r_{i}, r_{i}), \theta_{i}^{j}(r_{i}, r_{j})\geq \theta_{i}^{j}(r_{j}, r_{j});$\\
(2). if $r_{i}\geq r_{j}$, then $\theta_{i}^{j}(r_{i}, r_{j})\leq\theta_{i}^{j}(r_{i}, r_{i}), \theta_{i}^{j}(r_{i}, r_{j})\leq\theta_{i}^{j}(r_{j}, r_{j})$.
\end{theorem}

\section{Hyperbolic background geometry}
In this section we give the proof of Theorem \ref{thm-character-H}. Ge-Hua-Zhou \cite{GHZ-2} proved the longtime existence, uniqueness and a upper bound estimate for the solution $r(t)$ to the combinatorial Ricci flow (\ref{equ-H}) with ideal circle patterns.
\begin{theorem}[Lemma 3.3, \cite{GHZ-2}]\label{existence-H}
For any initial ideal circle pattern $r(0)\in\mathbb{R}_{>0}^{N}$, the combinatorial Ricci flow (\ref{equ-H}) has a unique solution $r(t)$ which exists for all time $t\geq0$. Moreover, each $r_i(t)$ is bounded from above in $[0, +\infty)$.
\end{theorem}
Furthermore, we can prove a lower bound estimate for inner angles along the solution $r(t)$ to the combinatorial Ricci flow (\ref{equ-H}) under the character assumption:
$\mathcal{L}(\mathcal{T},\Phi)_i>2\pi$ for all $i\in V$.
\begin{theorem} \label{lower-H}
Let $G$ be the 1-skeleton of a cellular decomposition $\mathcal{D}$ of a closed surface $S$ with a weight $\Phi: E\rightarrow (0, \pi)$. Assume the character
$$\mathcal{L}(\mathcal{D},\Phi)_i>2\pi$$
for all $i\in V$. Let $r(t)$ be a solution to the combinatorial Ricci flow (\ref{equ-H}). Then there exists a constant $C>0$ depending on the weighted cellular decomposition $(\mathcal D, \Phi)$ such that $r_{i}(t)\geq C$ for all vertices.
\end{theorem}
We need the following calculus lemma introduced by Ge-Hua \cite{GH}. For a continuous function $f: [0, +\infty)\rightarrow \mathbb{R}$ and any $C\in \mathbb{R}$, the upper level set of $f$ at $C$ is defined as $\{f> C\}:=\{t\in [0, +\infty): f(t)> C\}$.
The lower level set $\{f< C\}$ is defined similarly.
\begin{lemma}[Lemma 3.9, \cite{GH}]
\label{calculus}
Let $f: [0, +\infty)\rightarrow \mathbb{R}$ be a locally Lipschitz function. Suppose that there is a constant $C_{1}$ such that
$f'(t)\leq0$,  for  a.e. $t\in \{f> C_{1}\}$, then $$f(t)\leq \max\{f(0), C_{1}\}, \forall t\in [0, +\infty).$$
Similarly, if $f'(t)\geq0$,  for  a.e. $t\in \{f< C_{1}\}$, then $$f(t)\geq \min\{f(0), C_{1}\}, \forall t\in [0, +\infty).$$
\end{lemma}
\begin{proof}
Set $g(t): =\min_{m\in V}r_{m}(t)$. Then $g(t)$ is a locally Lipschitz function and for a.e. $t\in [0, +\infty)$, there exists a special vertex $i\in V$ depending on $t$, such that
\begin{equation}\label{equ-hyper}
g(t)=r_{i}(t), g'(t)=r_{i}'(t).
\end{equation}
For the particular vertex $i$, by Proposition \ref{lim}, $\theta_{i}(t\vec{1})$ is continuously differentiable and
$$\textmd{lim}_{t\rightarrow 0}\theta^{j}_{i}(t, t)=\frac{\Phi_{ij}}{2}.$$
Hence, for any positive constant $\epsilon_0$, small enough and is to be determined later, there exists a constant $C>0$ such that for any $t\leq C$,
$$\theta^{j}_{i}(t, t)\geq \frac{\Phi_{ij}}{2}-\epsilon_{0}.$$
We claim that $g'(t)\geq0$,  for  a.e. $t\in \{g< C\}.$ \par
Let $t\in [0, +\infty)$ and $i\in V$ satisfying (\ref{equ-hyper}) and $t\in \{g< C\}$. Then for any $j\sim i$,
$r_{i}(t)=\min\{r_{i}(t), r_{j}(t)\}< C$. By Theorem \ref{comparison},
$$\theta^{j}_{i}(r_{i}(t), r_{j}(t))\geq\theta^{j}_{i}(r_{i}(t), r_{i}(t))\geq\frac{\Phi_{ij}}{2}-\epsilon_{0}.$$
Under the assumption $\mathcal{L}(\mathcal{D},\Phi)_j>2\pi$ for each vertex $j\in V$, we choose the constant $\epsilon_{0}$ so that
$$0<\epsilon_{0}<\frac{\min_{j\in V}\left(\mathcal{L}(\mathcal{D},\Phi)_j-2\pi\right)}{2\max_{j\in V}d_j},$$
where $d_{j}$ denotes the number of $j\in V$ satisfying $j\sim i$.
Note $\epsilon_{0}$ depends only on the data of the weighted cellular decomposition $(\mathcal D, \Phi)$ on $S$. It follows
\begin{equation*}
\begin{aligned}
K_{i}(r(t))=&\;2\pi-2\sum_{j\sim i}\theta^{j}_{i}(r_{i}(t), r_{j}(t))\\
\leq &\;2\pi-2\sum_{j\sim i}(\frac{\Phi_{ij}}{2}-\epsilon_{0})\\
=&\;2\epsilon_{0}\cdot d_i-(\mathcal{L}(\mathcal{D},\Phi)_i-2\pi)\\
\leq&\;2\epsilon_{0}\cdot\max_{j\in V}d_j-\min_{j\in V}\big(\mathcal{L}(\mathcal{D},\Phi)_j-2\pi\big)\\
<&\;0.
\end{aligned}
\end{equation*}
By (\ref{equ-hyper}) and the combinatorial Ricci flow (\ref{equ-H}),
$$g'(t)=r_{i}'(t)=-K_{i}\sinh r_{i}\geq 0.$$
This proves the claim. Then the theorem follows from the claim and Lemma \ref{calculus}.
\end{proof}

Let $r(t)$ be a solution to the combinatorial Ricci flow (\ref{equ-H}). Then by Theorem \ref{existence-H} and Theorem \ref{lower-H}, there exist positive constants $C_{1}$ and $C_{2}$ depending on the weighted cellular decomposition $(\mathcal D, \Phi)$ such that
$$C_{1}\leq r_{i}(t)\leq C_{2}, \forall i\in V, t\in [0, +\infty).$$
This is equivalent to say that $r(t)$ lies in the compact region in $\mathbb{R}^{N}_{>0}$, which implies that $r(t)$ converges as $t\rightarrow +\infty$ (See the proof of Theorem 1.7 in \cite{GHZ-2} for more details). Then the existence part of Theorem \ref{thm-character-H} is obtained by the following proposition proved by Ge-Hua-Zhou \cite{GHZ-2}.
\begin{proposition}[Proposition 3. 5, \cite{GHZ-2}]\label{convergence-H}
If the solution $r(t)$ to the flow (\ref{equ-H}) converges to a vector $r^{*}\in\mathbb{R}_{>0}^{N}$, then $r^{*}$ gives an ideal circle pattern with all vertex curvatures equal to zero.
\end{proposition}
Furthermore, by Theorem 1.7 of \cite{GHZ-2}, if the combinatorial Ricci flow (\ref{equ-H}) converges, then it converges exponentially fast. The uniqueness part of Theorem \ref{thm-character-H} is a consequence of the Bobenko-Springborn theorem \cite{BS} which states that the curvature map $r\mapsto K$ is injective. To sum up, we show that there exists a unique hyperbolic ideal circle pattern $r_{ze}$ (ideal circle pattern with zero curvature), and the solution $r(t)$ to combinatorial Ricci flow (\ref{equ-H}) converges exponentially fast to $r_{ze}$ for any initial ideal circle pattern $r(0)\in \mathbb{R}^{N}_{>0}$. \par
Finally, we prove the nonexistence part of Theorem \ref{thm-character-H}.
\begin{proof}
Assume $\mathcal{L}(\mathcal{D},\Phi)_i\leq2\pi$ at each vertex $i\in V$. Hence the average character $\mathcal{L}_{av}\leq2\pi$. By Proposition \ref{prop-euler-number} we get $\chi(S)\geq0$. The total curvature formula (\ref{GB}) in the hyperbolic background says:
$$\sum_{i\in V}K_i=2\pi\chi(S)+\text{Area}(S).$$
Hence there is no ideal circle pattern $r$ with curvature zero. We complete the proof of the nonexistence part of Theorem \ref{thm-character-H}.
\end{proof}

To complete the proof of Theorem \ref{thm-character-H}, we can even show more, that is, any initial ideal circle pattern shrinks to a point along the combinatorial Ricci flow (\ref{equ-H}) if $\mathcal{L}(\mathcal{D},\Phi)_i<2\pi$ for all vertices. Actually, we have
\begin{theorem}\label{thm-converge-zero}
Let $G$ be the 1-skeleton of a cellular decomposition $\mathcal{D}$ of a closed surface $S$ with a weight $\Phi: E\rightarrow (0, \pi)$. If the character $\mathcal{L}(\mathcal{D},\Phi)_i<2\pi$ at each vertex, then the solution $r(t)$ to the combinatorial Ricci flow (\ref{equ-H}) converges exponentially fast to $0$ when $t$ tends to $+\infty$.
\end{theorem}
\begin{proof}
We claim that if the character $\mathcal{L}(\mathcal{D},\Phi)_i<2\pi$ at each vertex, then for any ideal circle pattern $r$ there exists a vertex $i$ such that
$$K_{i}\geq C_{0},$$ where $C_{0}=C_{0}(\mathcal{D},\Phi)>0$ is a sufficiently small positive constant depending on the weighted cellular decomposition $(\mathcal D, \Phi)$.\par
Let $i$ be the vertex such that $r_{i}=\max_{j\in V}r_{j}$. For any $j\sim i$, we have $r_{i}=\max\{r_{i}, r_{j}\}$. By Theorem \ref{comparison},
\begin{equation}\label{com-H}
\theta^{j}_{i}(r_{i}, r_{j})\leq \theta^{j}_{i}(r_{i}, r_{i}).
\end{equation}
If the character $\mathcal{L}(\mathcal{T},\Phi)_{j}<2\pi$ for each vertex $j\in V$, we choose the constant $\epsilon_{0}$ depends only on the data of the weighted cellular decomposition $(\mathcal D, \Phi)$ on $S$ so that
$$0<\epsilon_{0}<\frac{\min_{j\in V}\left(2\pi-\mathcal{L}(\mathcal{D},\Phi)_j\right)}{2\max_{j\in V}d_j}.$$
Set $$C_{0}=\min_{j\in V}\big(2\pi-\mathcal{L}(\mathcal{D},\Phi)_j\big)-2\epsilon_{0}\cdot\max_{j\in V}d_j>0.$$
By Proposition \ref{lim}, $\theta^{j}_{i}(t, t)$ is continuously differentiable and
$$\lim_{t\rightarrow 0}\theta^{j}_{i}(t, t)=\frac{\Phi_{ij}}{2}.$$
Therefore, for the positive constant $\epsilon_0$, there exists a constant $C>0$ such that for any $t\leq C$,
$$\theta^{j}_{i}(t, t)\leq \frac{\Phi_{ij}}{2}+\epsilon_{0}.$$
Choose any $t_{0}$ satisfying that $0<t_{0}<\min\{r_i, C\}$. Combining (\ref{com-H}) and Proposition \ref{lim} which states that $t\mapsto \theta^{j}_{i}(t, t)$ is strictly decreasing we have
\begin{equation}\label{ine}
\theta^{j}_{i}(r_{i}, r_{j})\leq \theta^{j}_{i}(r_{i}, r_{i})<\theta^{j}_{i}(t_{0}, t_{0})\leq \frac{\Phi_{ij}}{2}+\epsilon_{0}.
\end{equation}
It follows
\begin{equation*}
\begin{aligned}
K_{i}(r)=&\;2\pi-2\sum_{j\sim i}\theta^{j}_{i}(r_{i}, r_{j})\\
\geq &\;2\pi-2\sum_{j\sim i}(\frac{\Phi_{ij}}{2}+\epsilon_{0})\\
=&\;(2\pi-\mathcal{L}(\mathcal{D},\Phi)_i)-2\epsilon_{0}\cdot d_i\\
\geq&\;C_{0}.
\end{aligned}
\end{equation*}
This completes the proof of the claim.
Let $r(t)$ be a solution to the combinatorial Ricci flow (\ref{equ-H}).
Set $r_{M}(t):=\max_{i\in V}r_{i}(t)$. Then for a.e. $t\in [0, +\infty)$, there exists $i\in V$ depending on $t$, such that $r_{M}(t)=r_{i}(t)$ and $r'_{M}(t)=r_{i}'(t)$.
By the claim above, for the ideal circle pattern $r(t)$, we have $K_{i}(r(t))\geq C_{0}$, where $C_{0}>0$ is a sufficiently small positive constant depending on weighted cellular decomposition $(\mathcal D, \Phi)$. By the combinatorial Ricci flow (\ref{equ-H}), for a.e. $t\in [0, +\infty)$
$$r'_{M}(t)=r'_{i}(t)\leq -C_{0}\sinh(r_{i}(t))=-C_{0}\sinh(r_{M}(t)),$$
which is equivalent to $(\ln\tanh\frac{r_{M}(t)}{2})'\leq -C_{0}$.
Integrate both sides from $0$ to $t$, $\tanh\frac{r_{M}(t)}{2}\leq \tanh\frac{r_{M}(0)}{2}e^{-C_{0}t}.$
Hence $r(t)$ converges exponentially fast to $0$ as $t\rightarrow \infty$.
This completes the proof of Theorem \ref{thm-converge-zero}.
\end{proof}

\section{Eucldean background geometry}
\hspace{14pt}
In this section we give the proof of Theorem \ref{thm-character-E}. Ge-Hua-Zhou \cite{GHZ-2} proved the longtime existence, uniqueness and the invariance proposition for the solution $r(t)$ to the combinatorial Ricci flow (\ref{equ-E}) with ideal circle patterns.

\begin{theorem}[\cite{GHZ-2}, Lemma 4.2]\label{existence-E}
For any $r(0)\in R^{N}_{>0}$, the combinatorial Ricci flow (\ref{equ-E}) has a unique solution $r(t)$ which exists for all time $t\geq0$. Moreover, $\prod_{i=1}^{N}r_{i}(t)$ remains a constant $C=\prod_{i=1}^{N}r_{i}(0)$ along the solution $r(t)$ to the combinatorial Ricci flow (\ref{equ-E}).
\end{theorem}

We need the following calculus lemma introduced by Ge-Hua \cite{GH}. For a continuous function $f: [0, +\infty)\rightarrow \mathbb{R}$ and any $C\in \mathbb{R}$, the upper level set of $f$ at $C$ is defined as
$$\{f> C\}:=\{t\in [0, +\infty): f(t)> C\}.$$
The lower level set $\{f< C\}$ is defined similarly.

\begin{lemma}[\cite{GH}, Lemma 3.9] \label{calculus}
Let $f: [0, +\infty)\rightarrow \mathbb{R}$ be a locally Lipschitz function. Suppose that there is a constant $C$ such that
$f'(t)\leq0$,  for  a.e. $t\in \{f> C\}$, then $$f(t)\leq \max\{f(0), C\}, \forall t\in [0, +\infty).$$
Similarly, if $f'(t)\geq0$,  for  a.e. $t\in \{f< C\}$, then $$f(t)\geq \min\{f(0), C\}, \forall t\in [0, +\infty).$$
\end{lemma}

Next we prove the lower bound estimate for the solution $r(t)$ to the combinatorial Ricci flow (\ref{equ-E}).
\begin{theorem} \label{lower-E}
Let $(\mathcal{D},\Phi)$ be a weighted cellular decomposition of a closed surface $S$. Assume the character
$$\mathcal{L}(\mathcal{D},\Phi)_i\geq2\pi(1-\frac{\chi(S)}{N})$$
for all $i\in V$. Let $r(t)$ be a solution to the combinatorial Ricci flow (\ref{equ-E}). Then there exists a positive constant $C_{1}>0$ depending on the initial data $r(0)$ such that $r_{i}(t)\geq C_{1}$ for all vertices.
\end{theorem}
\begin{proof}
Set $g(t): =\min_{m\in V}r_{m}(t)$. Then $g(t)$ is a locally Lipschitz function and for a.e. $t\in [0, +\infty)$, there exists a special vertex $i\in V$ depending on $t$, such that
\begin{equation}\label{equ-0}
g(t)=r_{i}(t), g'(t)=r_{i}'(t).
\end{equation}
For the particular vertex $i$, by Proposition \ref{lim}, $\theta_{i}^{j}(t, t)$ satisfies the following equality
$$\theta_{i}^{j}(t, t)\equiv \frac{\Phi_{ij}}{2}, \forall t>0.$$
Choosing any fixed sufficient large positive constant $C_{1}>0$, we claim that $g'(t)\geq0$,  for  a.e. $t\in \{g< C_{1}\}.$ \par
Let $t\in [0, +\infty)$ and $i\in V$ satisfying (\ref{equ-0}) and $t\in \{g< C_{1}\}$. Then for any triangle $\triangle(v_{i}v_{j}\tau_{k})$ in two-circle configurations realized by the ideal circle pattern $r(t)$,
$$r_{i}(t)=\min\{r_{i}(t), r_{j}(t)\}< C_{1}.$$
Set $\vec{r}(t)=(r_{i}(t), r_{j}(t))$, then by Theorem \ref{comparison} we have $$\theta_{i}^{j}(\vec{r}(t))\geq\theta_{i}^{j}(r_{i}(t), r_{i}(t)).$$
Moreover, under the assumption $\mathcal{L}(\mathcal{D},\Phi)_i\geq2\pi(1-\frac{\chi(S)}{N})$ for each vertex $i\in V$, we have
\begin{equation*}
\begin{aligned}
K_{av}-K_{i}(r(t))=&\;\frac{2\pi\chi(X)}{N}-2\pi+2\sum_{j\sim i}\theta_{i}^{j}(\vec{r}(t))\\
\geq &\;\frac{2\pi\chi(X)}{N}-2\pi+2\sum_{j\sim i}\frac{\Phi_{ij}}{2}\\
=&\;\mathcal{L}(\mathcal{D},\Phi)_i-2\pi(1-\frac{\chi(S)}{N})\\
\geq&\;0.
\end{aligned}
\end{equation*}
By (\ref{equ-0}) and the combinatorial Ricci flow (\ref{equ-E}),
$$g'(t)=r_{i}'(t)=(K_{av}-K_{i})r_{i}\geq 0.$$
This proves the claim. Then the theorem follows from the claim and Lemma \ref{calculus}.
\end{proof}

Furthermore, there also exists a positive upper bound for a solution $r(t)$ to the combinatorial Ricci flow (\ref{equ-E}) by combining Theorem \ref{existence-E} and Theorem \ref{lower-E}.
\begin{corollary}\label{upper-E}
Let $(\mathcal{D},\Phi)$ be a weighted cellular decomposition of a closed surface $S$. Assume the character
$$\mathcal{L}(\mathcal{D},\Phi)_i\geq2\pi(1-\frac{\chi(S)}{N})$$
for all $i\in V$. Let $r(t)$ be a solution to the combinatorial Ricci flow (\ref{equ-E}). Then there exists a positive constant $C_{2}>0$ depending on the initial data $r(0)$ such that $r_{i}(t)\leq C_{2}$ for all vertices.
\end{corollary}
\begin{proof}
By Theorem \ref{lower-E}, there exists a constant $C_{1}>0$ depending on the initial data $r(0)$ such that $r_{i}(t)\geq C_{1}, i=1, \cdots, N$.
Set $g(t): =\max_{m\in V}r_{m}(t)$. Then $g(t)$ is a locally Lipschitz function and for a.e. $t\in [0, +\infty)$, there exists a special vertex $i\in V$ depending on $t$, such that
\begin{equation}\label{equ2}
g(t)=r_{i}(t).
\end{equation}
Then by Theorem \ref{existence-E} we have
\begin{equation}\label{equ2}
g(t)=r_{i}(t)=\frac{\prod_{i=1}^{N}r_{i}(t)}{\prod_{j\neq i}^{N}r_{j}(t)}=\frac{\prod_{i=1}^{N}r_{i}(0)}{\prod_{j\neq i}^{N}r_{j}(t)}\leq\frac{\prod_{i=1}^{N}r_{i}(0)}{C_{1}^{N-1}}=C_{2},
\end{equation}
which completes the proof.
\end{proof}

Let $r(t)$ be a solution to the combinatorial Ricci flow (\ref{equ-E}). Then by Theorem \ref{lower-E} and Corollary \ref{upper-E}, there exist positive constants $C_{1}$ and $C_{2}$ depending on the weighted cellular decomposition $(\mathcal D, \Phi)$ such that
$$C_{1}\leq r_{i}(t)\leq C_{2}, \forall i\in V, t\in [0, +\infty).$$
This is equivalent to say that $r(t)$ lies in the compact region in $\mathbb{R}^{N}_{>0}$, which implies that $r(t)$ converges as $t\rightarrow +\infty$ (See the proof of Theorem 1.8 in \cite{GHZ-2} for more details). Then the existence part of Theorem \ref{thm-character-E} is obtained by the following proposition proved by Ge-Hua-Zhou \cite{GHZ-2}.
\begin{proposition}[Proposition 4.3, \cite{GHZ-2}]\label{convergence-E}
If the solution $r(t)$ to the flow (\ref{equ-E}) converges to a vector $r^{*}\in\mathbb{R}_{>0}^{N}$, then $r^{*}$ gives an ideal circle pattern with all vertex curvatures equal to $K_{av}$.
\end{proposition}
Furthermore, by Theorem 1.8 of \cite{GHZ-2}, if the combinatorial Ricci flow (\ref{equ-E}) converges, then it converges exponentially fast. The uniqueness part of Theorem \ref{thm-character-H} is a consequence of the Bobenko-Springborn theorem \cite{BS} which states that the curvature map $r\mapsto K$ is injective up to scalings (see Theorem 1.5 in \cite{GHZ-2} for more details). To sum up, we show that there exists a unique constant ideal circle pattern $r_{av}$ (ideal circle pattern with constant curvature $K_{av}$), and the solution $r(t)$ to the combinatorial Ricci flow (\ref{equ-E}) converges exponentially fast to $r_{av}$ for any initial ideal circle pattern $r(0)\in \mathbb{R}^{N}_{>0}$.\par
Finally, we prove the nonexistence part of Theorem \ref{thm-character-E}. In fact, we can even show more, that is, any initial ideal circle pattern shrinks to a point along the combinatorial Ricci flow (\ref{equ-E}).
\begin{theorem}\label{thm-converge-E}
Let $G$ be the 1-skeleton of a cellular decomposition $\mathcal{D}$ of a closed surface $S$ with a weight $\Phi: E\rightarrow (0, \pi)$.  If the character $\mathcal{L}(\mathcal{D},\Phi)_i<2\pi(1-\frac{\chi(S)}{N})$ at each vertex, then there exists no constant ideal circle patterns. Moreover, the solution $r(t)$ to the combinatorial Ricci flow (\ref{equ-E}) converges exponentially fast to $0$ when $t\rightarrow +\infty$.
\end{theorem}
\begin{proof}
We claim that if the character $\mathcal{L}(\mathcal{D},\Phi)_i<2\pi(1-\frac{\chi(S)}{N})$ at each vertex, then for any ideal circle pattern $r$ there exists a vertex $i$ such that
$$K_{i}-K_{av}\geq C_{0},$$ where $C_{0}>0$ is a positive constant depending only on the weighted cellular decomposition $(\mathcal D, \Phi)$.\par
Let $i$ be the vertex such that $r_{i}=\max_{j\in V}r_{j}$. For any $j\sim i$, we have $r_{i}=\max\{r_{i}, r_{j}\}$. By Theorem \ref{comparison},
\begin{equation}\label{com}
\theta^{j}_{i}(r_{i}, r_{j})\leq \theta^{j}_{i}(r_{i}, r_{i}).
\end{equation}
If the character $\mathcal{L}(\mathcal{D},\Phi)_j<2\pi(1-\frac{\chi(S)}{N})$ for each vertex $j\in V$, we choose the constant $\epsilon_{0}$ depends only on the data of the weighted cellular decomposition $(\mathcal D, \Phi)$ on $S$ so that
$$0<\epsilon_{0}<\frac{\min_{j\in V}\big(2\pi(1-\frac{\chi(S)}{N})-\mathcal{L}(\mathcal{D},\Phi)_j\big)}{2\max_{j\in V}d_j}.$$
Set $$C_{0}=\min_{j\in V}\big(2\pi(1-\frac{\chi(S)}{N})-\mathcal{L}(\mathcal{D},\Phi)_j\big)-2\epsilon_{0}\cdot\max_{j\in V}d_j>0.$$
By Proposition \ref{lim}, $\theta^{j}_{i}(t, t)$ is a constant, i.e.
\begin{equation}\label{inv}
\theta^{j}_{i}(t, t)\equiv \frac{\Phi_{ij}}{2}, \forall t>0.
\end{equation}
Hence for the positive constant $\epsilon_{0}>0$,
$$\theta^{j}_{i}(t, t)< \frac{\Phi_{ij}}{2}+\epsilon_{0}, \forall t>0.$$
By (\ref{com}) and (\ref{inv} we have
\begin{equation*}
\begin{aligned}
K_{i}-K_{av}=&\;2\pi-2\sum_{j\sim i}\theta^{j}_{i}(r_{i}, r_{j})-\frac{2\pi\chi(S)}{N}\\
\geq &\;2\pi-2\sum_{j\sim i}(\frac{\Phi_{ij}}{2}+\epsilon_{0})-\frac{2\pi\chi(S)}{N}\\
=&\;\big(2\pi(1-\frac{\chi(S)}{N})-\mathcal{L}(\mathcal{D},\Phi)_i\big)-2\epsilon_{0}\cdot d_i\\
\geq&\;C_{0}.
\end{aligned}
\end{equation*}
This completes the proof of the claim.
Let $r(t)$ be a solution to the combinatorial Ricci flow (\ref{equ-E}).
Set $r_{M}(t):=\max_{i\in V}r_{i}(t)$. Then for a.e. $t\in [0, +\infty)$, there exists $i\in V$ depending on $t$, such that $r_{M}(t)=r_{i}(t)$ and $r'_{M}(t)=r_{i}'(t)$.
By the claim above, we have $K_{i}(r(t))-K_{av}\geq C_{0}$, where $C_{0}>0$ is a positive constant depending on weighted cellular decomposition $(\mathcal D, \Phi)$. By the combinatorial Ricci flow (\ref{equ-E}), for a.e. $t\in [0, +\infty)$
$$r'_{M}(t)=r'_{i}(t)\leq -C_{0}r_{i}(t)=-C_{0}r_{M}(t),$$
which is equivalent to $(\ln r_{M}(t))'\leq -C_{0}$.
Integrate both sides from $0$ to $t$, $r_{M}(t)\leq r_{M}(0)e^{-C_{0}t}.$
Hence $r(t)$ converges exponentially fast to $0$ as $t\rightarrow \infty$.
This completes the proof of Theorem \ref{thm-converge-E}.
\end{proof}
\section{The prescribed flow}
\hspace{14pt}
Recall $\bar{K}=(\bar{K_1}, \cdots, \bar{K_N})$ is a prescribed vertex curvature and we can consider the prescribed circle pattern problem. The prescribed ideal Ricci flow in the hyperbolic background geometry is
\begin{equation}\label{e-prescribe-H}
\frac{dr_i}{dt}=(\bar{K_i}-K_i)\sinh r_i
\end{equation}
and the prescribed ideal Ricci flow in the Euclidean background geometry is
\begin{equation}\label{e-prescribe-E}
\frac{dr_i}{dt}=(\bar{K_i}-K_i)r_i.
\end{equation}
\subsection{Proof of Theorem \ref{prescribed-H}}
\hspace{14pt}
First we can prove Theorem \ref{prescribed-H} as follows.
\begin{proof}
Theorem 3.7 in \cite{GHZ-2} shows that there exists a unique longtime solution $r(t), t\in [0,\infty)$ to the prescribed ideal Ricci flow (\ref{e-prescribe-H}).
To show the convergence part, similar to the proof of Theorem \ref{thm-character-H}, we just need to show $r(t)\subset \subset\mathbb{R}^N_{>0}$.
We show this by proving the following Claim 1 and Claim 2.\\
\textbf{Claim 1}: If $\bar{K_i}<2\pi$, $\forall i\in V$, then all $r_i(t)$ are bounded from above uniformly.

By Proposition \ref{lim}, $\theta^{j}_{i}(t, t)$ is continuously differentiable and
$$\textmd{lim}_{t\rightarrow +\infty}\theta^{j}_{i}(t, t)=0.$$
Hence for any $0<\epsilon_{0}<\min_i(2\pi-\bar{K_i})/(2\max_i d_i)$, there is a positive constant $C=C(\mathcal{D},\Phi, \epsilon_{0})>0$ so that whenever $t\geq C$, then $\theta^{j}_{i}(t, t)\leq \epsilon_{0}$.
Set $g(t): =\max_{m\in V}r_{m}(t)$. Then $g(t)$ is a locally Lipschitz function and for a.e. $t\in [0, +\infty)$, there exists a special vertex $i\in V$ depending on $t$, such that
\begin{equation}\label{equ-prescribed-H}
g(t)=r_{i}(t), g'(t)=r_{i}'(t).
\end{equation}
We claim that $g'(t)\leq0$,  for  a.e. $t\in \{g> C\}.$ \par
Let $t\in [0, +\infty)$ and $i\in V$ satisfying (\ref{equ-prescribed-H}) and $t\in \{g> C\}$. Then for any $j\sim i$,
$r_{i}(t)=\max\{r_{i}(t), r_{j}(t)\}>C$. By Theorem \ref{comparison},
$$\theta^{j}_{i}(r_{i}(t), r_{j}(t))\leq\theta^{j}_{i}(r_{i}(t), r_{i}(t))\leq\epsilon_{0}.$$
 It follows at each vertex $i \in V$,
$$\bar{K_i}-K_i=\bar{K_i}-2\pi+2\sum_{j\sim i}\theta^{j}_{i}(r_{i}(t), r_{j}(t))
\leq\bar{K_i}-2\pi+2d_i\epsilon_{0}<0.$$
Then by the prescribed ideal Ricci flow in hyperbolic background geometry
\begin{equation*}
g'(t)=r'_i(t)=(\bar{K_i}-K_i)\sinh r_i<0.
\end{equation*}
This proves the claim. Then Claim 1 follows from this claim and Lemma \ref{calculus}.\\
\textbf{Claim 2}: If $\bar{K_i}>2\pi-\mathcal{L}(\mathcal{D},\Phi)_i$, $\forall i\in V$, then all $r_i(t)$ have a uniform positive lower bound.

By Proposition \ref{lim}, $\theta^{j}_{i}(t, t)$ is continuously differentiable and
$$\textmd{lim}_{t\rightarrow 0}\theta^{j}_{i}(t, t)=\frac{\Phi_{ij}}{2}.$$
Hence for any $0<\epsilon_{0}<\min_i(\bar{K_i}-(2\pi-\mathcal{L}(\mathcal{T},\Phi)_i))/(2\max_i d_i)$, there is a positive constant $\tilde{C}=\tilde{C}(\mathcal{D},\Phi, \epsilon_{0})>0$ so that whenever $t\leq \tilde{C}$, then $\theta^{j}_{i}(t, t)\geq \frac{\Phi_{ij}}{2}-\epsilon_{0}$.
Set $g(t): =\min_{m\in V}r_{m}(t)$. Then $g(t)$ is a locally Lipschitz function and for a.e. $t\in [0, +\infty)$, there exists a special vertex $i\in V$ depending on $t$, such that
\begin{equation}\label{equ1}
g(t)=r_{i}(t), g'(t)=r_{i}'(t).
\end{equation}
We claim that $g'(t)\geq0$,  for  a.e. $t\in \{g< \tilde{C}\}.$ \par
Let $t\in [0, +\infty)$ and $i\in V$ satisfying (\ref{equ1}) and $t\in \{g<\tilde{C}\}$. Then for any $j\sim i$,
$r_{i}(t)=\min\{r_{i}(t), r_{j}(t)\}<\tilde{C}$. By Theorem \ref{comparison},
$$\theta^{j}_{i}(r_{i}(t), r_{j}(t))\geq\theta^{j}_{i}(r_{i}(t), r_{i}(t))\geq\frac{\Phi_{ij}}{2}-\epsilon_{0}.$$
 It follows at each vertex $i \in V$,
 \begin{equation*}
\begin{aligned}
\bar{K_i}-K_i=&\;\bar{K_i}-2\pi+2\sum_{j\sim i}\theta^{j}_{i}(r_{i}(t), r_{j}(t))\\
\geq&\;\bar{K_i}-2\pi+2\sum_{j\sim i}(\frac{\Phi_{ij}}{2}-\epsilon_{0})\\
=&\;\bar{K_i}-(2\pi-\mathcal{L}(\mathcal{D},\Phi)_i)-2d_i\epsilon_{0}\\
>&\;0.
\end{aligned}
\end{equation*}
Then by the prescribed ideal Ricci flow (\ref{e-prescribe-H})
\begin{equation*}\label{equ-prescribe-H}
g'(t)=r'_i(t)=(\bar{K_i}-K_i)\sinh r_i>0.
\end{equation*}
This proves the claim. Then Claim 2 follows from this claim and Lemma \ref{calculus}.\par
Finally we show the nonconvergence part. Assume $\bar{K_i}\leq2\pi-\mathcal{L}(\mathcal{D},\Phi)_i$ at each vertex $i\in V$, which is equivalent to $\mathcal{L}(\mathcal{D},\Phi)_i\leq2\pi-\bar{K_i}$ at each vertex $i\in V$. Hence the average character $\mathcal{L}_{av}\leq2\pi-(\sum_{i\in V}\bar{K_i})/N$. By Proposition \ref{prop-euler-number} we get $\sum_{i\in V}\bar{K_i}\leq2\pi\chi(S)$. The total curvature formula (\ref{GB}) in the hyperbolic background geometry says:
$$\sum_{i\in V}K_i=2\pi\chi(S)+\text{Area}(S)>2\pi\chi(S).$$
Hence there is no ideal circle pattern $\bar{r}$ with curvature $K(\bar{r})=\bar{K}$. Similar to the proof of Theorem \ref{thm-converge-zero}, we can show that
for any initial ideal circle pattern, the solution $r(t)$ to the prescribed flow (\ref{eq-prescribe-H}) converges exponentially fast to $0$ when $t\rightarrow +\infty$ if $\bar{K_i}<2\pi-\mathcal{L}(\mathcal{D},\Phi)_i$ for all vertices. We omit the details and complete the proof of Theorem \ref{prescribed-H}.
\end{proof}

\subsection{Proof of Theorem \ref{prescribed-E}}
\hspace{14pt}
In this subsection we give the proof of Theorem \ref{prescribed-E}. In the Euclidean background geometry, first we assume the prescribed vertex curvature $\bar{K}=(\bar{K_1}, \cdots, \bar{K_N})$ satisfies the following equality
\begin{equation}\label{prescribe-sum}
\sum_{i=1}^{N}\bar{K}_{i}=2\pi\chi(S).
\end{equation}
Then we extend Theorem \ref{existence-E} and establish the invariance proposition for the solution $r(t)$ to the prescribed combinatorial Ricci flow (\ref{e-prescribe-E}) with ideal circle patterns.
\begin{proposition}\label{prescribed-inv}
Let $r(t)$ be a solution to the prescribed combinatorial Ricci flow (\ref{e-prescribe-E}) with ideal circle patterns. Then $\prod_{i=1}^{N}r_{i}(t)$ remains a constant $C=\prod_{i=1}^{N}r_{i}(0)$ along the solution $r(t)$ to the prescribed combinatorial Ricci flow (\ref{e-prescribe-E}).
\end{proposition}
\begin{proof}
By a coordinate transformation, set $u_{i}=\ln r_{i}$ then we can equivalently rewrite the prescribed combinatorial Ricci flow (\ref{e-prescribe-E}) as
\begin{equation}\label{e-prescribe-E-u}
    u'_{i}(t)=\frac{du_{i}}{dt}=\bar{K}_{i}-K_{i}.
\end{equation}
Suppose that $r(t)=(r_{1}(t), \cdots, r_{N}(t)), t\in [0, +\infty)$ is a solution to the prescribed combinatorial Ricci flow (\ref{e-prescribe-E}), then $u_{i}(t)=\ln r_{i}(t)(i=1, \cdots, N)$ is a solution to (\ref{e-prescribe-E-u}). By (\ref{prescribe-sum}) and the total curvature formula (\ref{GB}) in the Euclidean background geometry
$$\sum_{i\in V}K_i=2\pi\chi(S),$$
we have
$$\frac{d(\sum_{i=1}^{N}u_{i}(t))}{dt}=\sum_{i=1}^{N}u'_{i}(t)=\sum_{i=1}^{N}(\bar{K}_{i}-K_{i})=2\pi\chi(S)-2\pi\chi(S)=0.$$
This implies that $\sum_{i=1}^{N}u_{i}(t)$ remains a constant value $C=\sum_{i=1}^{N}u_{i}(0)$ for any $t\in [0, +\infty)$, equivalently, $$\prod_{i=1}^{N}r_{i}(t)\equiv\prod_{i=1}^{N}r_{i}(0), \forall t\in [0, +\infty).$$
\end{proof}

Next we can prove Theorem \ref{prescribed-E} as follows.
\begin{proof}
Theorem 4.5 in \cite{GHZ-2} shows that there exists a unique longtime solution $r(t), t\in [0,\infty)$ to the prescribed ideal Ricci flow (\ref{e-prescribe-E}).
To show the convergence part, similar to the proof of Theorem \ref{thm-character-E}, we just need to show $r(t)\subset \subset\mathbb{R}^N_{>0}$.
We show this by proving the following Claim 1 and Claim 2.\\
\textbf{Claim 1}: If $\bar{K_i}\geq2\pi-\mathcal{L}(\mathcal{D},\Phi)_i$, $\forall i\in V$, then all $r_i(t)$ have a uniform positive lower bound.\par

Set $g(t): =\min_{m\in V}r_{m}(t)$. Then $g(t)$ is a locally Lipschitz function and for a.e. $t\in [0, +\infty)$, there exists a special vertex $i\in V$ depending on $t$, such that
\begin{equation}\label{equ-priscribe-g}
g(t)=r_{i}(t), g'(t)=r_{i}'(t).
\end{equation}
For the particular vertex $i$, by Proposition \ref{lim}, $\theta_{i}^{j}(t, t)$ satisfies the following equality
$$\theta_{i}^{j}(t, t)\equiv \frac{\Phi_{ij}}{2}, \forall t>0.$$
Choosing any fixed sufficient large positive constant $C_{1}>0$, we claim that $g'(t)\geq0$,  for  a.e. $t\in \{g< C_{1}\}.$ \par
Let $t\in [0, +\infty)$ and $i\in V$ satisfying (\ref{equ-priscribe-g}) and $t\in \{g< C_{1}\}$. Then for any triangle $\triangle(v_{i}v_{j}\tau_{k})$ in two-circle configurations realized by the ideal circle pattern $r(t)$,
$$r_{i}(t)=\min\{r_{i}(t), r_{j}(t)\}< C_{1}.$$
Set $\vec{r}(t)=(r_{i}(t), r_{j}(t))$, then by Theorem \ref{comparison} we have $$\theta_{i}^{j}(\vec{r}(t))\geq\theta_{i}^{j}(r_{i}(t), r_{i}(t)).$$
Moreover, under the assumption $\bar{K}_{i}\geq 2\pi-\mathcal{L}(\mathcal{D},\Phi)_i$ for each vertex $i\in V$, we have
\begin{equation*}
\begin{aligned}
\bar{K}_{i}-K_{i}(r(t))=&\;\bar{K}_{i}-2\pi+2\sum_{j\sim i}\theta_{i}^{j}(\vec{r}(t))\\
\geq &\;\bar{K}_{i}-2\pi+2\sum_{j\sim i}\frac{\Phi_{ij}}{2}\\
=&\;\bar{K}_{i}-(2\pi-\mathcal{L}(\mathcal{D},\Phi)_i)\\
\geq&\;0.
\end{aligned}
\end{equation*}
By (\ref{equ-priscribe-g}) and the prescribed combinatorial Ricci flow (\ref{e-prescribe-E}),
$$g'(t)=r_{i}'(t)=(\bar{K}_{i}-K_{i})r_{i}\geq 0.$$
This proves the claim. Then the theorem follows from the claim and Lemma \ref{calculus}.\\
\textbf{Claim 2}: If $\bar{K_i}\geq2\pi-\mathcal{L}(\mathcal{D},\Phi)_i$, $\forall i\in V$, then all $r_i(t)$ have a uniform positive upper bound.\par

By Claim 1, there exists a constant $C_{1}>0$ depending on the initial data $r(0)$ such that $r_{i}(t)\geq C_{1}, i=1, \cdots, N$.
Set $g(t): =\max_{m\in V}r_{m}(t)$. Then $g(t)$ is a locally Lipschitz function and for a.e. $t\in [0, +\infty)$, there exists a special vertex $i\in V$ depending on $t$, such that $g(t)=r_{i}(t)$. Then by Proposition \ref{prescribed-inv} we have
\begin{equation}\label{equ7}
g(t)=r_{i}(t)=\frac{\prod_{i=1}^{N}r_{i}(t)}{\prod_{j\neq i}^{N}r_{j}(t)}=\frac{\prod_{i=1}^{N}r_{i}(0)}{\prod_{j\neq i}^{N}r_{j}(t)}\leq\frac{\prod_{i=1}^{N}r_{i}(0)}{C_{1}^{N-1}}=C_{2},
\end{equation}
which proves Claim 2.\par

Finally we show the nonconvergence part. Similar to the proof of Theorem \ref{thm-converge-E}, we can show that if $\bar{K_i}<2\pi-\mathcal{L}(\mathcal{D},\Phi)_i$ for all vertices, then there is not any ideal circle pattern $\bar{r}$ with curvature $K(\bar{r})=\bar{K}$. Moreover, the solution $r(t)$ to the prescribed flow (\ref{eq-prescribe-E}) converges exponentially fast to $0$ when $t\rightarrow +\infty$. We omit the details and complete the proof of Theorem \ref{prescribed-E}.
\end{proof}

\noindent Chang Li, chang\_li@ruc.edu.cn\\[2pt]
\emph{School of Mathematics, Renmin University of China, Beijing 100872, P. R. China}\\[2pt]

\noindent Aijin Lin, linaijin@nudt.edu.cn\\[2pt]
\emph{College of Science, National University of Defense Technology, Changsha 410073, P. R. China.}\\[2pt]

\noindent Liangming Shen, lmshen@buaa.edu.cn\\[2pt]
\emph{School of Mathematical Sciences, Beihang University, and Key Laboratory of Mathematics Informatics Behavioral Semantics,
Ministry of Education, Beijing 100191, P. R. China}

\end{document}